\documentclass[12pt]{article}
\usepackage{amsmath,amssymb,amsthm}
\numberwithin{equation}{section}
\usepackage{hyperref}
\usepackage{units}
\usepackage{color}
\usepackage[T1]{fontenc}
\usepackage[utf8]{inputenc}
\usepackage{authblk}
\usepackage{bm}
\usepackage{graphicx}

\usepackage{datetime}

\usepackage[toc,page]{appendix}

\newcommand{\Z}{{\mathbb Z}}

\newtheorem{thm}{Theorem}

\newtheorem{lemma}{Lemma}

\title{Weighted generalized Fibonacci sums}

\author[]{Kunle Adegoke \thanks{adegoke00@gmail.com}}

\affil{Department of Physics and Engineering Physics, \mbox{Obafemi Awolowo University}, 220005 Ile-Ife, Nigeria}

\begin{document}

\date{}

\maketitle

\begin{abstract}
\noindent We derive weighted sums, including binomial and double binomial sums, for the generalized Fibonacci sequence $\{G_m\}$ where for $m\ge 2$, $G_m=G_{m-1}+G_{m-2}$ with initial values $G_0$ and $G_1$.
\end{abstract}

\section{Introduction}
The generalized Fibonacci sequence $\{G_m\}$ is given, for $m\ge 2$, by $G_m=G_{m-1}+G_{m-2}$ with initial values $G_0$ and $G_1$. In particular, when $G_0 = 0$ and $G_1= 1$, we have the Fibonacci sequence $\{F_m\}$, and when $G_0= 2$ and $G_1= 1$, we have the sequence of Lucas numbers, $\{L_m\}$. Extension to negative index is provided through $G_{-m}=(-1)^m(F_{m+1}G_0-F_mG_1)$, Vajda~\cite{vajda}, identity~(9). Whenever an identity contains numbers from two generalized Fibonacci sequences, the numbers from one will be denoted by $G$ while the numbers from the other will be denoted by $H$, with the appropriate subscripts. 

\bigskip

The following identities connecting the Fibonacci numbers, the Lucas numbers and the generalized Fibonacci numbers are well known, Vajda~\cite{vajda}, identities~(8),~(10a),~(10b),~(18): 
\begin{equation}\label{eq.t3slb1g}
G_{m + n } = F_{n-1}G_m +F_nG_{m+1}\,,
\end{equation}
\begin{equation}\label{eq.l5rjizg}
G_{m + n} + (-1)^nG_{m-n} = L_nG_m\,,
\end{equation}
\begin{equation}\label{eq.bazmec8}
G_{m + n}  - ( - 1)^n G_{m - n}  = F_n (G_{m - 1}  + G_{m + 1} )\,,
\end{equation}
\begin{equation}\label{eq.llmg2cf}
G_{n + r} H_{m + n}  - G_n H_{m + n + r}  = ( - 1)^n (G_r H_m  - G_0 H_{m + r} )\,,
\end{equation}
and Howard~\cite{howard03}, Theorem~3.1, Corrolary~3.5:
\begin{equation}\label{eq.tpyn1ql}
( - 1)^rF_nG_m  = F_{n + r}G_{m + r}   -  F_r G_{m + n + r}\,.
\end{equation}
Note, by the way, that identity~\eqref{eq.t3slb1g} can be obtained from identity~\eqref{eq.tpyn1ql} through the set of transformations $m\to m+1$, $n\to-n$ and $r\to n-1$. 

\bigskip

Our purpose is to discover the weighted summation identities associated with the identities~\eqref{eq.bazmec8}, \eqref{eq.llmg2cf} and \eqref{eq.tpyn1ql} since those associated with identities~\eqref{eq.t3slb1g} and \eqref{eq.l5rjizg} are already contained as special cases, $p=1=-q$, of the results obtained in an earlier paper~\cite{adegoke18}.
\section{Weighted sums}
\begin{lemma}[\cite{adegoke18}, Lemma 2]\label{lem.s9jfs7n}
Let $\{X_m\}$ be any arbitrary sequence, where $X_m$, $m\in\Z$, satisfies a second order recurrence relation $X_m=f_1X_{m-a}+f_2X_{m-b}$, where $f_1$ and $f_2$ are arbitrary non-vanishing complex functions, not dependent on $m$, and $a$ and $b$ are integers. Then,
\begin{equation}\label{eq.awbhgnm}
f_2\sum_{j = 0}^k {f_1^j X_{m - b  - a j} }  = X_m  - f_1^{k + 1} X_{m - (k + 1)a }\,,
\end{equation}
\begin{equation}\label{eq.jjikwds}
f_1\sum_{j = 0}^k {f_2^j X_{m - a  - b j} }  = X_m  - f_2^{k + 1} X_{m - (k + 1)b }\,,
\end{equation}
\begin{equation}
\sum_{j = 0}^k {\frac{{X_{m + a  - (b  - a )j} }}{{( - f_1/f_2)^j }}}  = f_1X_m  + \frac{f_2}{{( - f_1/f_2)^k }}X_{m - (k + 1)(b  - a )}
\end{equation}
and
\begin{equation}
\sum_{j = 0}^k {\frac{{X_{m + b  - (a  - b )j} }}{{( - f_2/f_1)^j }}}  = f_2X_m  + \frac{f_1}{{( - f_2/f_1)^k }}X_{m - (k + 1)(a  - b )}\,,
\end{equation}
for $k$ any integer.
\end{lemma}
\begin{thm}\label{thm.mjeylda}
The following identities hold for integers $m$, $n$, $r$ and $k$:
\begin{equation}
F_r\sum_{j = 0}^k {( - 1)^{rj} \left( {\frac{{F_{n + r} }}{{F_n }}} \right)^j G_{m + n + r + rj} }  = ( - 1)^{kr} F_n\left(\frac{{F_{n+r} }}{{F_n }}\right)^{k+1}G_{m + (k + 1)r}  - ( - 1)^r F_nG_m,\quad n\ne 0\,,
\end{equation}
\begin{equation}
F_{n + r} \sum_{j = 0}^k {( - 1)^{(r + 1)j} \left( {\frac{{F_r }}{{F_n }}} \right)^j G_{m + r + (n + r)j} }  = ( - 1)^r F_n G_m  + ( - 1)^{(r + 1)k} F_n \left( {\frac{{F_r }}{{F_n }}} \right)^{k + 1} G_{m + (k + 1)(n + r)},\quad n\ne 0
\end{equation}
and
\begin{equation}
\sum_{j = 0}^k {\left( {\frac{{F_r }}{{F_{n + r} }}} \right)^j G_{m - r + nj} }  = ( - 1)^r \frac{{F_{n + r} }}{{F_n }}G_m  - ( - 1)^r \frac{{F_r }}{{F_n }}\left( {\frac{{F_r }}{{F_{n + r} }}} \right)^k G_{m + (k + 1)n},\quad n+r\ne 0\,.
\end{equation}

\end{thm}
\begin{proof}
In identity~\eqref{eq.tpyn1ql}, identify $f_1=(-1)^rF_{n+r}/F_n$, $f_2=-(-1)^rF_r/F_n$, $a=-r$, $b=-n-r$ and use these in Lemma~\ref{lem.s9jfs7n} with $X_m=G_m $.
\end{proof}
\section{Weighted binomial sums}
The following identities of Lucas~\cite{long88}:
\begin{equation}
\sum_{j = 0}^k {\binom kjF_{s + j} }  = F_{s + 2k} \mbox{ and } \sum_{j = 0}^k {\binom kjL_{s + j} }  = L_{s + 2k}\,,
\end{equation}
are the Fibonacci and Lucas specializations of the generalized Fibonacci sum:
\begin{equation}
\sum_{j = 0}^k {\binom kjG_{s + j} }  = G_{s + 2k}\,,
\end{equation}
which is itself a particular case of the following identity (Theorem~\ref{thm.sihtl16}, identity~\eqref{eq.m46aoy9}):
\[
\sum_{j = 0}^k {( - 1)^{rj} \binom kj\left( {\frac{{F_r }}{{F_n }}} \right)^j G_{m - rk + (n + r)j} }  = ( - 1)^{rk} \left( {\frac{{F_{n + r} }}{{F_n }}} \right)^k G_m ,\quad n\ne 0\,,
\]
being an evaluation at $r=2$, $n=-1$ and $m=s+2k$.

\bigskip

In this section, we will derive three general binomial summation identities for the $G-$sequence. Binomial sums for Fibonacci or Fibonacci-like sequences are also discussed in references~\cite{khan11, filliponi95, layman77, hoggatt64}. First we state a required Lemma.
\begin{lemma}[\cite{adegoke18}, Lemma 3]\label{lem.binomial}
Let $\{X_m\}$ be any arbitrary sequence. Let $X_m$, $m\in\Z$, satisfy a second order recurrence relation $X_m=f_1X_{m-a}+f_2X_{m-b}$, where $f_1$ and $f_2$ are non-vanishing complex functions, not dependent on $m$, and $a$ and $b$ are integers. Then,
\begin{equation}\label{eq.atln4oc}
\sum_{j = 0}^k {\binom kj\left( {\frac{f_2}{f_1}} \right)^j X_{m - ak  + (a  - b )j} }  = \frac{{X_m }}{{f_1{}^k }}\,,
\end{equation}
\begin{equation}\label{eq.dzwka2l}
\sum_{j = 0}^k {(-f_2)^j\binom kjX_{m + ak - bj}}  = f_1{}^k X_m
\end{equation}
and
\begin{equation}\label{eq.muqj9xn}
\sum_{j = 0}^k {(-f_1)^j\binom kjX_{m + bk - aj}}  = f_2{}^k X_m\,,
\end{equation}
for $k$ a non-negative integer.

\end{lemma}
\begin{thm}\label{thm.sihtl16}
The following identities hold for integers $n$, $m$ and $r$ and non-negative integer~$k$:
\begin{equation}\label{eq.rmtxz1n}
\sum_{j = 0}^k {( - 1)^j \binom kj\left( {\frac{{F_r }}{{F_{r + n} }}} \right)^j G_{m + rk + nj} }  = ( - 1)^{rk} \left( {\frac{{F_n }}{{F_{r + n} }}} \right)^k G_m ,\quad n+r\ne 0\,,
\end{equation}
\begin{equation}\label{eq.m46aoy9}
\sum_{j = 0}^k {( - 1)^{rj} \binom kj\left( {\frac{{F_r }}{{F_n }}} \right)^j G_{m - rk + (n + r)j} }  = ( - 1)^{rk} \left( {\frac{{F_{n + r} }}{{F_n }}} \right)^k G_m ,\quad n\ne 0\,,
\end{equation}
and
\begin{equation}\label{eq.f1t4x43}
\sum_{j = 0}^k {( - 1)^j ( - 1)^{rj} \binom kj\left( {\frac{{F_{n + r} }}{{F_n }}} \right)^j G_{m - (n + r)k + rj} }  = ( - 1)^k ( - 1)^{rk} \left( {\frac{{F_r }}{{F_n }}} \right)^k G_m ,\quad n\ne 0\,.
\end{equation}

\end{thm}
\begin{proof}
Use, in Lemma~\ref{lem.binomial}, the values of $a$, $b$, $f_1$ and $f_2$ found in the proof of Theorem~\ref{thm.mjeylda}, with $X_m=G_m $.
\end{proof}
A particular instance of identity~\eqref{eq.rmtxz1n} is
\begin{equation}
\sum_{j = 0}^k {( - 1)^j \binom kj\left( {\frac{{F_r }}{{F_{r + n} }}} \right)^j G_{nj} }  = \left( {\frac{{F_n }}{{F_{r + n} }}} \right)^k (F_{rk+1}G_{0}-F_{rk}G_1)\,,
\end{equation}
which at $r=n$ gives
\begin{equation}
\sum_{j = 0}^k {( - 1)^j \binom kj\frac{{G_{nj} }}{{L_n{}^j }}}  = \frac{F_{nk+1}G_{0}-F_{nk}G_1}{L_n{}^k}\,,
\end{equation}
since $F_{2n}=L_nF_n$. In particular, we have
\begin{equation}
\sum_{j = 0}^k {( - 1)^j \binom kj G_j}  = F_{k+1}G_{0}-F_{k}G_1\,.
\end{equation}

\bigskip

Similarly, from identities~\eqref{eq.m46aoy9} and \eqref{eq.f1t4x43}, we have
\begin{equation}\label{eq.mer4bue}
\sum_{j = 0}^k {( - 1)^{rj} \binom kj\left( {\frac{{F_r }}{{F_n }}} \right)^j G_{(n + r)j} }  = ( - 1)^{rk} \left( {\frac{{F_{n + r} }}{{F_n }}} \right)^k G_{rk} ,\quad n\ne 0\,,
\end{equation}
and
\begin{equation}\label{eq.kq31wfz}
\sum_{j = 0}^k {( - 1)^j ( - 1)^{rj} \binom kj\left( {\frac{{F_{n + r} }}{{F_n }}} \right)^j G_{rj} }  = ( - 1)^k ( - 1)^{rk} \left( {\frac{{F_r }}{{F_n }}} \right)^k G_{(n+r)k} ,\quad n\ne 0\,,
\end{equation}
which at $r=n$ give
\begin{equation}
\sum_{j = 0}^k {( - 1)^{nj} \binom kjG_{2nj} }  = ( - 1)^{nk} L_n{}^k G_{nk}
\end{equation}
and
\begin{equation}
\sum_{j = 0}^k {( - 1)^j ( - 1)^{nj} \binom kjL_n{}^j G_{nj} }  = ( - 1)^k ( - 1)^{nk} G_{2nk}\,.
\end{equation}
In particular, we have
\begin{equation}
\sum_{j = 0}^k {( - 1)^j \binom kjG_{2j} }  = ( - 1)^k G_k
\end{equation}
and
\begin{equation}
\sum_{j = 0}^k {\binom kjG_j }  = G_{2k}\,.
\end{equation}
\section{Weighted double binomial sums}
\begin{lemma}[\cite{adegoke18c}, Lemma 5]\label{lem.h2de9i7}
Let $\{X_m\}$ be any arbitrary sequence, $X_m$ satisfying a third order recurrence relation $X_m=f_1X_{m-a}+f_2X_{m-b}+f_3X_{m-c}$, where $f_1$, $f_2$ and $f_3$ are arbitrary non-vanishing functions and $a$, $b$ and $c$ are integers. Let $k$ be a non-negative integer. Then, the following identities hold:
\begin{equation}\label{eq.wgx2r2f}
\sum_{j = 0}^k {\sum_{s = 0}^j {\binom kj\binom js\left( {\frac{{f_2 }}{{f_3 }}} \right)^j\left( {\frac{{f_1 }}{{f_2 }}} \right)^s  X_{m - ck + (c - b)j + (b - a)s} } }  = \frac{{X_m }}{{f_3{}^k }}\,,
\end{equation}
\begin{equation}\label{eq.sm9bygb}
\sum_{j = 0}^k {\sum_{s = 0}^j {\binom kj\binom js\left( {\frac{{f_3 }}{{f_2 }}} \right)^j\left( {\frac{{f_1 }}{{f_3 }}} \right)^s  X_{m - bk + (b - c)j + (c - a)s} } }  = \frac{{X_m }}{{f_2{}^k }}\,,
\end{equation}
\begin{equation}
\sum_{j = 0}^k {\sum_{s = 0}^j {\binom kj\binom js\left( {\frac{{f_3 }}{{f_1 }}} \right)^j\left( {\frac{{f_2 }}{{f_3 }}} \right)^s X_{m - ak + (a - c)j + (c - b)s} } }  = \frac{{X_m }}{{f_1{}^k }}\,,
\end{equation}
\begin{equation}
\sum_{j = 0}^k {\sum_{s = 0}^j {\binom kj\binom js\left( {\frac{{f_2 }}{{f_3 }}} \right)^j\left( {-\frac{{1 }}{{f_2 }}} \right)^s X_{m - (c-a)k + (c - b)j + bs } } }  = \left(-\frac {f_1}{f_3}\right)^kX_m\,,
\end{equation}
\begin{equation}
\sum_{j = 0}^k {\sum_{s = 0}^j {\binom kj\binom js\left( {\frac{{f_1 }}{{f_3 }}} \right)^j\left( {-\frac{{1 }}{{f_1 }}} \right)^s X_{m - (c-b)k + (c - a)j + as } } }  = \left(-\frac {f_2}{f_3}\right)^kX_m
\end{equation}
and
\begin{equation}\label{eq.o7540wl}
\sum_{j = 0}^k {\sum_{s = 0}^j {\binom kj\binom js\left( {\frac{{f_1 }}{{f_2 }}} \right)^j\left( {-\frac{{1 }}{{f_1 }}} \right)^s X_{m  - (b-c)k + (b - a)j + as} } }  = \left(-\frac {f_3}{f_2}\right)^kX_m\,.
\end{equation}

\end{lemma}

Note that each of identities~\eqref{eq.wgx2r2f}--~\eqref{eq.o7540wl} can be written in the following respective equivalent form:
\begin{equation}
\sum_{j = 0}^k {\sum_{s = 0}^{k - j} {\binom kj\binom {k-j}s\left( {\frac{{f_3 }}{{f_1 }}} \right)^j \left( {\frac{{f_2 }}{{f_1 }}} \right)^s X_{m - ak - (c - a)j - (b - a)s} } }  = \frac{{X_m }}{{f_1 ^k }}\,,
\end{equation}
\begin{equation}
\sum_{j = 0}^k {\sum_{s = 0}^{k - j} {\binom kj\binom {k-j}s\left( {\frac{{f_2 }}{{f_1 }}} \right)^j \left( {\frac{{f_3 }}{{f_1 }}} \right)^s X_{m - ak - (b - a)j - (c - a)s} } }  = \frac{{X_m }}{{f_1 ^k }}\,,
\end{equation}
\begin{equation}
\sum_{j = 0}^k {\sum_{s = 0}^{k - j} {\binom kj\binom {k-j}s\left( {\frac{{f_1 }}{{f_2 }}} \right)^j \left( {\frac{{f_3 }}{{f_2 }}} \right)^s X_{m - bk - (a - b)j - (c - b)s} } }  = \frac{{X_m }}{{f_2 ^k }}\,,
\end{equation}
\begin{equation}
\sum_{j = 0}^k {\sum_{s = 0}^{k - j} {\binom kj\binom {k-j}s( - 1)^{j + s} f_3 ^j f_2 ^s X_{m + ak - cj - bs} } }  = f_1 ^k X_m \,,
\end{equation}
\begin{equation}
\sum_{j = 0}^k {\sum_{s = 0}^{k - j} {\binom kj\binom {k-j}s( - 1)^{j + s} f_3 ^j f_1 ^s X_{m + bk - cj - as} } }  = f_2 ^k X_m 
\end{equation}
and
\begin{equation}
\sum_{j = 0}^k {\sum_{s = 0}^{k - j} {\binom kj\binom {k-j}s( - 1)^{j + s} f_2 ^j f_1 ^s X_{m + ck - bj - as} } }  = f_3 ^k X_m\,. 
\end{equation}

\begin{thm}
The following identities hold for integers $n$, $m$, $r$ and non-negative integer~$k$:
\begin{equation}
\sum_{j = 0}^k {\sum_{s = 0}^{k - j} {( - 1)^{n(j + s)} \binom kj\binom {k-j}sF_n ^{j + s} G_{m - 2nk + (n + 1)j + (n - 1)s} } }  = ( - 1)^{nk} G_m\,,
\end{equation}
\begin{equation}
\sum_{j = 0}^k {\sum_{s = 0}^{k - j} {( - 1)^{n(j + s)} \binom kj\binom {k-j}sF_n ^{j + s} G_{m - 2nk + (n - 1)j + (n + 1)s} } }  = ( - 1)^{nk} G_m\,,
\end{equation}
\begin{equation}
\sum_{j = 0}^k {\sum_{s = 0}^{k - j} {( - 1)^{nj} \binom kj\binom {k-j}s\frac{{G_{m - (n + 1)k - (n - 1)j + 2s} }}{{F_n ^j }}} }  = \frac{{G_m }}{{F_n ^k }},\quad n\ne 0\,,
\end{equation}
\begin{equation}
\sum_{j = 0}^k {\sum_{s = 0}^j {( - 1)^s \binom kj\binom js\frac{{G_{m + (n + 1)k - 2j + (n + 1)s} }}{{F_n ^s }}} }  = ( - 1)^{(n + 1)k} \frac{{G_m }}{{F_n ^k }},\quad n\ne 0\,,
\end{equation}
\begin{equation}
\sum_{j = 0}^k {\sum_{s = 0}^j {( - 1)^{n(j + s) + s} \binom kj\binom js\frac{{G_{m + 2k - (n + 1)j + 2ns} }}{{F_n ^j }}} }  = ( - 1)^k G_m,\quad n\ne 0\,,
\end{equation}
and
\begin{equation}
\sum_{j = 0}^k {\sum_{s = 0}^j {( - 1)^{n(j + s) + s} \binom kj\binom js\frac{{G_{m - 2k - (n - 1)j + 2ns} }}{{F_n ^j }}} }  = ( - 1)^k G_m,\quad n\ne 0\,.
\end{equation}

\end{thm}
\begin{proof}
Write identity~\eqref{eq.bazmec8} as $G_{m}  = ( - 1)^n G_{m - 2n}  + F_n G_{m - n - 1}  + F_nG_{m - n + 1} $, identify $a=2n$, $b=n+1$, $c=n-1$, $f_1=( - 1)^n$, $f_2=F_n=f_3$ and use these in Lemma~\ref{lem.h2de9i7} with $X_m=G_m$.
\end{proof}
\begin{thm}
The following identities hold for integers $n$, $m$ and $r$ and non-negative integer~$k$:
\begin{equation}
\sum_{j = 0}^k {\sum_{s = 0}^j {( - 1)^{nj + s} \binom kj\binom js\frac{{G_{n + r} ^{j - s} G_n ^s }}{{G_0 ^j }}H_{m + rk + (n - r)j + rs} } }  = \left( {\frac{{G_r }}{{G_0 }}} \right)^k H_m,\quad G_0\ne 0\,, 
\end{equation}
\begin{equation}
\sum_{j = 0}^k {\sum_{s = 0}^j {( - 1)^{n(j + s) + s} \binom kj\binom js\frac{{G_0 ^{j - s} G_n ^s }}{{G_{n + r} ^j }}H_{m + nk + (r - n)j + ns} } }  = ( - 1)^{nk} \left( {\frac{{G_r }}{{G_{n + r} }}} \right)^k H_m,\quad G_{n+r}\ne 0\,,
\end{equation}
\begin{equation}
\sum_{j = 0}^k {\sum_{s = 0}^j {( - 1)^{n(j + s) + j} \binom kj\binom js\frac{{G_0 ^{j - s} G_{n + r} ^s }}{{G_n ^j }}H_{m + (n + r)k - nj + (n - r)s} } }  = ( - 1)^{(n + 1)k} \left( {\frac{{G_r }}{{G_n }}} \right)^k H_m,\quad G_n\ne 0\,, 
\end{equation}
\begin{equation}
\sum_{j = 0}^k {\sum_{s = 0}^{k - j} {( - 1)^{ns + j + s} \binom kj\binom {k-j}s\frac{{G_0 ^j G_{n + r} ^s }}{{G_r ^{j + s} }}H_{m - (n + r)k + rj + ns} } }  = ( - 1)^{(n + 1)k} \left( {\frac{{G_n }}{{G_r }}} \right)^k H_m,\quad G_r\ne 0\,,
\end{equation}
\begin{equation}
\sum_{j = 0}^k {\sum_{s = 0}^{k - j} {( - 1)^{ns + j} \binom kj\binom {k-j}s\frac{{G_0 ^j G_n ^s }}{{G_r ^{j + s} }}H_{m - nk + rj + (n + r)s} } }  = ( - 1)^{nk} \left( {\frac{{G_{n + r} }}{{G_r }}} \right)^k H_m,\quad G_r\ne 0\,,
\end{equation}
and
\begin{equation}
\sum_{j = 0}^k {\sum_{s = 0}^{k - j} {( - 1)^{n(j + s) + j} \binom kj\binom {k-j}s\frac{{G_{n + r}^j G_n ^s }}{{G_r^{j + s} }}H_{m - rk + nj + (n + r)s} } }  = \left( {\frac{{G_0 }}{{G_r }}} \right)^k H_m,\quad G_r\ne 0\,.
\end{equation}

\end{thm}
\begin{proof}
Write identity~\eqref{eq.llmg2cf} as $G_r H_m  =  - ( - 1)^n G_n H_{m + n + r}  + ( - 1)^n G_{n + r} H_{m + n}  + G_0 H_{m + r} $, identify $a =  - n - r$, $b =  - n$, $c =  - r$, $f_1  =  - ( - 1)^n G_n /G$, $f_2  = ( - 1)^n G_{n + r} /G_r $, $f_3  = G_0 /G_r $ and use these in Lemma~\ref{lem.h2de9i7} with $X_m=H_m$.
\end{proof}


\begin{thebibliography}{99}

\bibitem{adegoke18} K.~Adegoke, Weighted sums of some second-order sequences, \emph{arXiv:1803.09054[math.NT]} (2018).

\bibitem{adegoke18c} K.~Adegoke, Weighted Tribonacci sums, \emph{arXiv:1804.06449[math.CA]} (2018).

\bibitem{filliponi95} P.~Filipponi, Some binomial Fibonacci identities, \emph{The Fibonacci Quarterly} {\bf 33}:3 (1995), 251--257.

\bibitem{hoggatt64} V.~ E.~Hoggatt and Jr.~and  M.~Bicknell, Some new Fibonacci identities, \emph{The Fibonacci Quarterly} {\bf 2}:1 (1964), 29--32.

\bibitem{howard03} F.~T.~Howard, The  sum  of the  squares  of two generalized  Fibonacci  numbers., \emph{The  Fibonacci  Quarterly  } {\bf 41}:1 (2003), 80--84.

\bibitem{khan11} M.~A.~Khan and H.~Kwong , Some binomial identities associated with the generalized natural number sequence, \emph{The Fibonacci Quarterly} {\bf 49}:1 (2011), 57--65.

\bibitem{layman77} J.~W.~Layman, Certain general binomial-Fibonacci sums, \emph{The Fibonacci Quarterly} {\bf 15}:4 (1977), 362--366.

\bibitem{long88} C.~T.~Long, Some binomial Fibonacci identities, \emph{Applications of Fibonacci numbers} {\bf 3} (1988), 241--254.

\bibitem{vajda} S.~Vajda, \emph{Fibonacci and Lucas numbers, and the golden section: theory and applications}, Dover Press, (2008).



\end{thebibliography}
\end{document}